\documentclass{amsart}
\usepackage{amsmath, amssymb, mathrsfs, verbatim, multirow}
\usepackage[pdftex,
            pdfauthor={M. Ghasemi, M. Marshall},
            pdftitle={Lower Bounds for a Polynomial on a basic closed semialgebraic set using geometric programming},
            pdfsubject={polynomial optimization},
            pdfkeywords={Positive Polynomials, Sums of Squares, Geometric Programming},
            pdfproducer={Latex with hyperref},
            pdfcreator={PDFLaTeX},
            colorlinks=true]{hyperref}
\newcommand{\mathsym}[1]{{}}

\newtheorem{thm}{Theorem}[section]
\newtheorem{lemma}[thm]{Lemma}

\theoremstyle{definition}
\newtheorem{exm}[thm]{Example}

\newtheorem{rem}[thm]{Remark}

\begin{document}

\title[Lower bounds for a polynomial]
 {Lower Bounds for a Polynomial on a basic closed semialgebraic set using geometric programming}
 \date{}

\author{Mehdi Ghasemi}
\author{Murray Marshall}
\address{Department of Mathematics and Statistics,
University of Saskatchewan,
Saskatoon, \newline \indent
SK S7N 5E6, Canada}
\email{mehdi.ghasemi@usask.ca, marshall@math.usask.ca}

\keywords{Positive polynomials, Sums of Squares, Geometric Programming}
\subjclass[2010]{Primary 12D99 Secondary 14P99, 90C22}

\begin{abstract} Let
$f,g_1,\dots,g_m$ be elements of the polynomial ring $\mathbb{R}[x_1,\dots,x_n]$. The paper deals with the general
problem of computing a lower bound for $f$ on the subset of $\mathbb{R}^n$ defined by the inequalities $g_i\ge 0$,
$i=1,\dots,m$. The paper shows that there is an algorithm for computing such a lower bound, based on geometric
programming, which applies in a large number of cases. 
The algorithm extends and generalizes earlier algorithms of Ghasemi and Marshall, dealing with the case $m=0$, and of
Ghasemi, Lasserre and Marshall, dealing with the case $m=1$ and $g_1= M-(x_1^d+\dots+x_n^d)$. Here, $d$ is required
to be an even integer $d \ge \max\{2,\deg(f)\}$. The algorithm is implemented in a SAGE program developed by the first author.
The bound obtained is typically not as good as the bound obtained using semidefinite programming, but it has the advantage
that it is computable rapidly, even in cases where the bound obtained by semidefinite programming is not computable.

\end{abstract}

\maketitle
%
\section{Introduction}
%
Let $f,g_1,\dots,g_m$ be elements of the polynomial ring $\mathbb{R}[\mathbf{x}]=\mathbb{R}[x_1,\cdots,x_n]$ and let
\[
	K_{\mathbf{g}} := \{ \mathbf{x} \in \mathbb{R}^n : g_j(\mathbf{x})\ge 0, \, j=1,\dots,m\}.
\]
Here,  $\mathbf{g} := (g_1,\dots, g_m)$. We refer to $K_{\mathbf{g}}$ as the basic closed semialgebraic set generated by
$\mathbf{g}$. Observe that if $m=0$, then $\mathbf{g}=\emptyset$ and $K_{\mathbf{g}}= \mathbb{R}^n$. Let
\[
	f_{*,\mathbf{g}}:= \inf\{ f(\mathbf{x}) : \mathbf{x} \in K_{\mathbf{g}}\}.
\]
One would like to have a simple algorithm for computing a lower bound for $f$ on $K_{\mathbf{g}}$, i.e., a lower bound
for $f_{*,\mathbf{g}}$. Lasserre's algorithm \cite{lasserre1} is such an algorithm. It produces a hierarchy of lower bounds
\[
	f_{\operatorname{sos},\mathbf{g}}^{(t)} = \sup \{ r \in \mathbb{R} : f-r = \sum_{j=0}^m \sigma_jg_j, \sigma_j \in \sum \mathbb{R}[\mathbf{x}]^2, \deg(\sigma_jg_j)\le t, j=0,\dots,m \}
\]
for $f$ on $K_{\mathbf{g}}$, one for each integer $t\ge \max\{ \deg(f), \deg(g_j): j=1,\dots,m\}$, which are computable
by semidefinite programming. Here, $g_0:=1$ and $\sum \mathbb{R}[\mathbf{x}]^2$ denotes the set of elements of
$\mathbb{R}[\mathbf{x}]$ which are sums of squares. Denote by $d$ the least even integer $d \ge \max\{ 2,\deg(f), \deg(g_j) : j=1,\dots,m\}$.
The algorithm in \cite{gha-mar2} deals with the case $m=0$, producing a lower bound $f_{\operatorname{gp}}$ for $f$
on $\mathbb{R}^n$ computable by geometric programming.\footnotemark\footnotetext{A function $\phi: (0,\infty)^n\rightarrow\mathbb{R}$ of the form $\phi(\underline{x})=cx_1^{a_1}\cdots x_n^{a_n}$,
where $c>0$, $a_i\in\mathbb{R}$ and $\underline{x}=(x_1,\ldots,x_n)$ is called a \it monomial function. \rm A sum of monomial functions
is called a \it posynomial function. \rm An optimization problem of the form
\[
	\left\lbrace
	\begin{array}{ll}
		\textrm{Minimize } & \phi_0(\underline{x})  \\
		\textrm{Subject to} & \phi_i(\underline{x})\leq 1, \ i=1,\ldots,m \text{ and } \  \psi_j(\underline{x})=1, \ j=1,\ldots,p
	\end{array}
	\right.
\]
where $\phi_0,\ldots,\phi_m$ are posynomials and $\psi_1,\ldots,\psi_p$ are monomial functions, is called a \textit{geometric program}. See \cite{boyd} and \cite{gha-mar2}.}
See \cite{fidalgo}, \cite{gha-mar1} and \cite{lasserre} for
precursors of \cite{gha-mar2}.  The algorithm in \cite{gha-lass-mar} is a variation of the algorithm in \cite{gha-mar2},
which deals with the case $m=1$, $g_1 = M-(x_1^{d}+\dots+x_n^{d})$, i.e., it produces a lower bound for $f$ on the
hyperellipsoid
\[
	B_{M} := \{ \mathbf{x} \in \mathbb{R}^n : x_1^{d}+\dots+x_n^{d}\le M\}.
\]
Again, this lower bound is computable by geometric programming. Of course, if $K_{\mathbf{g}}$ is compact, then
$K_{\mathbf{g}} \subseteq B_M$ for $M$ sufficiently large, so the lower bound established in \cite{gha-lass-mar} also
provides a lower bound for $f$ on $K_{\mathbf{g}}$.

Although the bounds obtained in \cite{gha-lass-mar} and \cite{gha-mar2} are typically not as good as the bounds obtained
in \cite{lasserre1}, the computation is much faster, especially when the coefficients are sparse, and problems where the number
of variables and the degree are large (problems where the method in \cite{lasserre1} breaks down completely) can be handled easily.

The goal of the present paper is to establish
a general lower bound for $f$ on $K_{\mathbf{g}}$ computable by geometric
programming. In case $m=0$ it should be the lower bound $f_{\operatorname{gp}}$ obtained in \cite{gha-mar2}. In case $m=1$ and $g_1 = M-(x_1^d+\dots+x_n^d)$, it should be the lower bound obtained in \cite{gha-lass-mar}.
This goal is not attained in every case, but it is attained in a large number of cases.

The idea is the following: Let $G(\lambda) = f-\sum_{j=1}^m \lambda_jg_j$ where
$\lambda = (\lambda_1,\dots,\lambda_m) \in [0,\infty)^m$. By \cite{gha-mar2} (also see Theorem \ref{thm1} below), $G(\lambda)_{\operatorname{gp}}$ is a lower
bound for $G(\lambda)$ on $\mathbb{R}^n$. It follows that $G(\lambda)_{\operatorname{gp}}$ is a lower bound for $f$ on
$K_{\mathbf{g}}$ and consequently, that
\[
	s(f,\mathbf{g}):=\sup\{ G(\lambda)_{\operatorname{gp}} : \lambda \in [0,\infty)^m\}
\]
is a lower bound for $f$ on $K_{\mathbf{g}}$. By \cite{gha-mar2}, for  each
$\lambda \in [0,\infty)^m$, $G(\lambda)_{\operatorname{gp}}$ is computable by geometric programming.
Unfortunately, this does not imply that the supremum is so computable, although there are important cases where it is;
see Theorem \ref{important special case} (2). More to the point, there are important cases where, even though the supremum
itself may not be computable by geometric programming, there is a relaxation which is computable by geometric programming;
see Theorem \ref{change variables} and Theorem \ref{important special case} (1).

The main new result
is Theorem \ref{important special case}.
See Remark \ref{extensions and clarifications} 
(6)  for the application of Theorem \ref{important special case} to
the computation of a lower bound on 
any product of hyperellipsoids. See Remark \ref{extensions and clarifications} (8), (9) and (10) for runtime 
and relative error computations. 
Theorem \ref{m=1} explains how the hypothesis of Theorem \ref{important special case} can be weakened slightly in the case $m=1$.
See Theorems \ref{m=2} and \ref{A=I} for other variants of Theorem \ref{important special case}. See Remark \ref{applications}
for some indication of how Theorems \ref{important special case}, \ref{m=1}, \ref{m=2} and \ref{A=I} can be applied in practice.
See Example \ref{examples} for sample computations. Theorem \ref{hypercube case} relates the lower bound on the hypercube
$\prod_{j=1}^n [-N_j,N_j]$ described in Remark \ref{extensions and clarifications} (4) to the trivial lower bound introduced in \cite[Section 3]{gha-lass-mar}.
The source code of a SAGE program, developed by the first author, which computes the lower bound of $f$ on $K_{\mathbf{g}}$
described in Theorem \ref{change variables}, is available at \href{https://github.com/mghasemi/CvxAlgGeo}{github.com/mghasemi/CvxAlgGeo}.
\section{the case $m=0$}
We recall the algorithm established in \cite{gha-mar2}. We need some notation. Fix an even integer $d\ge 2$. Let
$\mathbb{N}^n_d:=\{\alpha\in\mathbb{N}^n:\vert\alpha\vert\leq d\}$ where $\vert\alpha\vert=\sum_i\alpha_i$ for every
$\alpha\in\mathbb{N}^n$. Let $\epsilon_i:=(\delta_{i1},\cdots,\delta_{in})\in\mathbb{N}^n$, with $\delta_{ij}=d$ if
$i=j$ and $0$ otherwise and, given $f= \sum f_{\alpha} \mathbf{x}^{\alpha}\in\mathbb{R}[\mathbf{x}]$, $\deg(f)\le d$, let:
\[
\begin{array}{lcl}
\Omega(f)&:=&\{\alpha\in\mathbb{N}^n_{d}\::\:f_\alpha\neq0\}\setminus\{0,\epsilon_1,\cdots,\epsilon_n\}\\
\Delta(f)&:=&\{\alpha\in\Omega(f)\::\:f_\alpha\,\mathbf{x}^\alpha\mbox{ is not a square in }\mathbb{R}[\mathbf{x}]\}\\
\Delta(f)^{< d}&:=&\{\alpha\in\Delta(f)\::\:\vert\alpha\vert <d\}\\
\Delta(f)^{= d}&:=&\{\alpha\in\Delta(f)\::\:\vert\alpha\vert =d\}.
\end{array}
\]
Denote the coefficient $f_{\epsilon_i}$ by $f_{d,i}$ for $i=1,\dots,n$. One is most interested in the case where $\deg(f)=d$.
\begin{thm} \cite[Theorem 3.1]{gha-mar2}
\label{thm1}
Let $f\in\mathbb{R}[\mathbf{x}]$, $\deg(f)\le d$, and let $\rho(f)$ denote the optimal value of the program:
\begin{equation}
\label{lw}
\left\{
\begin{array}{lll}
Minimize & \sum\limits_{\alpha\in\Delta(f)^{<d}}(d-\vert\alpha\vert)
\left[\left(\frac{f_\alpha}{d}\right)^{d}\,\left(\frac{\alpha}{\mathbf{z}_\alpha}\right)^\alpha\right]^{1/(d-\vert\alpha\vert)} & \\
\mbox{s.t.} & \sum\limits_{\alpha\in\Delta(f)}z_{\alpha,i}\le f_{d,i}, & i=1,\ldots,n\\
& \left(\frac{\mathbf{z}_\alpha}{\alpha}\right)^\alpha =\left(\frac{f_\alpha}{d}\right)^{d}, & \alpha\in\Delta(f)^{=d}
\end{array}\right.
\end{equation}
where, for every $\alpha\in\Delta(f)$, the unknowns $\mathbf{z}_\alpha = (z_{\alpha,i})\in[0,\infty)^n$ satisfy $z_{\alpha,i}= 0$
if and only if $\alpha_i=0$. Then $f-f(0)+\rho(f)$ is a sum of binomial squares. In particular,
$f_{\operatorname{gp}} := f(0)-\rho(f)$ is a lower bound for $f$ on $\mathbb{R}^n$.
\end{thm}
Here, $\left(\frac{\alpha}{\mathbf{z}_\alpha}\right)^\alpha := \prod_{i=1}^n \frac{\alpha_i^{\alpha_i}}{(z_{\alpha,i})^{\alpha_i}}$
and $\left(\frac{\mathbf{z}_\alpha}{\alpha}\right)^\alpha := \prod_{i=1}^n \frac{(z_{\alpha,i})^{\alpha_i}}{\alpha_i^{\alpha_i}}$,
the convention being that $0^0 = 1$.

In (\ref{lw}), the constraint $\left(\frac{\mathbf{z}_\alpha}{\alpha}\right)^\alpha = \left(\frac{f_\alpha}{d}\right)^{d}$
can be replaced by the weaker constraint $\left(\frac{\mathbf{z}_\alpha}{\alpha}\right)^\alpha \ge \left(\frac{f_\alpha}{d}\right)^{d}$,
for each $\alpha \in \Delta(f)^{=d}$. If $\mathbf{z}$ is a feasible point for the latter program then, by shrinking suitably the
$z_{\alpha,i}$, $\alpha \in \Delta(f)^{=d}$, one gets a feasible point $\mathbf{z}'$ for the former program such that the objective
function of (\ref{lw}) evaluated at $\mathbf{z}$ and $\mathbf{z}'$ are the same.

If the feasible set of the program (\ref{lw}) is empty, then $\rho(f) = \infty$ and $f_{\operatorname{gp}} = -\infty$.
A sufficient (but not necessary) condition for the feasible set of (\ref{lw}) to be nonempty is that $\Delta(f)^{=d} =\emptyset$
and $f_{d,i}>0$, $i=1,\dots,n$. If $\deg(f)<d$ then either $\Delta(f)=\emptyset$ and $f_{\operatorname{gp}}=f(0)$ or
$\Delta(f)\ne \emptyset$ and $f_{\operatorname{gp}}=-\infty$.

If $f_{d,i}>0$, $i=1,\dots,n$ then (\ref{lw}) is a geometric program. Somewhat more generally, if $\forall$ $i=1,\dots,n$
either ($f_{d,i}>0$) or ($f_{d,i}=0$ and  $\alpha_i=0$ $\forall$ $\alpha \in \Delta(f)$), then (\ref{lw}) is a geometric program.
In the remaining cases (\ref{lw}) is not a geometric program and the feasible set of (\ref{lw}) is empty.
\section{general case}
We return to the set-up considered in the introduction, i.e., $f,g_1,\dots,g_m \in \mathbb{R}[\mathbf{x}]$,
$K_{\mathbf{g}} = \{ \mathbf{x} \in \mathbb{R}^n : g_j(\mathbf{x}) \ge 0, j=1,\dots,m\}$. We denote by $d$ the least even integer
$d \ge \max\{ 2,\deg(f), \deg(g_j) : j=1,\dots,m\}$.
We define $G(\lambda) = f-\sum_{j=1}^m \lambda_jg_j$, $\lambda = (\lambda_1,\dots,\lambda_m)\in [0,\infty)^m$.
Note that $G(\lambda)_{\alpha} = f_{\alpha}-\sum_{j=1}^m \lambda_j(g_j)_{\alpha}$,
$G(\lambda)_{d,i} = f_{d,i}-\sum_{j=1}^m \lambda_j(g_j)_{d,i}$, and $G(\lambda)(0) = f(0)-\sum_{j=1}^m \lambda_jg_j(0)$.
As explained already, $s(f,\mathbf{g}):=\sup\{ G(\lambda)_{\operatorname{gp}} : \lambda \in [0,\infty)^m\}$ is a lower
bound for $f$ on $K_{\mathbf{g}}$.\footnote{In fact, $s(f,\mathbf{g}) \le f_{\operatorname{sos},\mathbf{g}}^{(d)}$.
By Theorem \ref{thm1}, $G(\lambda)-G(\lambda)_{\operatorname{gp}}$ is a sum of binomial squares (obviously of degree at most $d$)
for each $\lambda \in [0,\infty)^m$. This implies that $G(\lambda)_{\operatorname{gp}} \le f_{\operatorname{sos},\mathbf{g}}^{(d)}$
for each $\lambda \in [0,\infty)^m$.} Let $\Delta := \Delta(f)\cup \Delta(-g_1)\cup \dots \cup \Delta(-g_m)$,
$\Delta^{<d} := \{ \alpha \in \Delta : |\alpha|<d\}$, $\Delta^{=d} := \{ \alpha \in \Delta : |\alpha|=d\}$.

It is convenient to define $g_0:= -f$, $\lambda_0 :=1$, so $G(\lambda) = -\sum_{j=0}^m \lambda_jg_j$.
\textit{We also assume from now on that $\Omega(-g_j)= \Delta(-g_j)$ for each $j=0,\dots,m$}. One can reduce to this case by ignoring all terms corresponding to elements of $\Omega(-g_j)\backslash \Delta(-g_j)$, i.e., by replacing $g_j$ by
\[
	g_j' := g_j(0)+\sum_{\alpha \in \Delta(-g_j)} (g_j)_{\alpha} \mathbf{x}^{\alpha}+ \sum_{i=1}^n (g_j)_{d,i}x_i^d, \ j=0,\dots,m.
\]
Then $-g_j' \le -g_j$ on $\mathbb{R}^n$, $j=0,\dots,m$, so $K_{\mathbf{g}} \subseteq K_{\mathbf{g}'}$ where
$\mathbf{g}' :=(g_1',\dots,g_m')$ and the minimum of $-g_0'$ on $K_{\mathbf{g}'}$ is not greater than the minimum of $-g_0$ on $K_{\mathbf{g}}$.

We consider the following program:
\begin{equation}
\label{lw1}
\left\{
\begin{array}{llr}
	\textrm{Minimize} & \multicolumn{2}{c}{\sum\limits_{j=1}^m \lambda_jg_j(0)+ \sum\limits_{\alpha\in\Delta^{<d}}(d-\vert\alpha\vert)
	 \left[\left(\frac{G(\lambda)_\alpha}{d}\right)^{d}\,\left(\frac{\alpha}{\mathbf{z}_\alpha}\right)^\alpha
	\right]^{1/(d-\vert\alpha\vert)}}\\
	\textrm{\mbox{s.t.}} & \sum\limits_{\alpha\in\Delta}z_{\alpha,i} \le G(\lambda)_{d,i}, & i=1,\ldots,n\\
	& \left(\frac{\mathbf{z}_\alpha}{\alpha}\right)^\alpha \ge \left(\frac{G(\lambda)_\alpha}{d}\right)^{d}, & \alpha\in\Delta^{=d}
\end{array}\right.
\end{equation}
where, for every $\alpha\in\Delta$, the unknowns $\mathbf{z}_\alpha = (z_{\alpha,i})\in[0,\infty)^n$ satisfy $z_{\alpha,i}= 0$
if and only if $\alpha_i=0$, and the unknowns $\lambda = (\lambda_1,\dots,\lambda_m)$ satisfy $\lambda_j\ge 0$.
\begin{thm}\label{general}
Denote by $\rho$ the optimum value of (\ref{lw1}). Then $f(0)-\rho$ is a lower bound for $s(f,\mathbf{g})$.
\end{thm}
\begin{proof} Note that $\Delta(G(\lambda)) \subseteq \Delta$ for each $\lambda \in [0,\infty)^m$. For suppose
$\alpha \in \Delta(G(\lambda))$. If $2\nmid \alpha$ then $G(\lambda)_{\alpha} \ne 0$, so $f_{\alpha}\ne 0$ or
$(g_j)_{\alpha} \ne 0$ for some $j$. If $2\mid \alpha$ then $G(\lambda)_{\alpha}<0$ so $f_{\alpha}<0$ or $(g_j)_{\alpha}>0$
for some $j$. Note also that if $(\mathbf{z},\lambda)$ is a feasible point of (\ref{lw1}), then
\[
	\sum\limits_{\alpha\in\Delta^{<d}}(d-\vert\alpha\vert)\left[
	 \left(\frac{G(\lambda)_\alpha}{d}\right)^{d}\,\left(\frac{\alpha}{\mathbf{z}_\alpha}\right)^\alpha
	\right]^{1/(d-\vert\alpha\vert)}
\]
is an upper bound for $\rho(G(\lambda))$, so
\[
	f(0)- \sum_{j=1}^m \lambda_jg_j(0)- \sum\limits_{\alpha\in\Delta^{<d}}(d-\vert\alpha\vert)\left[
	 \left(\frac{G(\lambda)_\alpha}{d}\right)^{d}\,\left(\frac{\alpha}{\mathbf{z}_\alpha}\right)^\alpha
	\right]^{1/(d-\vert\alpha\vert)}
\]
is a lower bound for $G(\lambda)_{\operatorname{gp}} =G(\lambda)(0)-\rho(G(\lambda))$, for each feasible point
$(\mathbf{z},\lambda)$ of (\ref{lw1}). It follows that $f(0)-\rho$ is a lower bound for $s(f,\mathbf{g})$.
\end{proof}
A sufficient (but not necessary) condition for the feasible set of (\ref{lw1}) to be nonempty is that
$\Delta^{=d} =\emptyset$ and there exists $\lambda \in [0,\infty)^m$ such that $G(\lambda)_{d,i}>0$, $i=1,\dots,n$.

Unfortunately, (\ref{lw1}) is generally not a geometric program, even if one replaces the constraint $\lambda_j\ge 0$
by $\lambda_j>0$, for $j=1,\dots,m$.\footnote{If one replaces the constraints $\lambda_j\ge 0$ by $\lambda_j>0$,
for $j=1,\dots,m$,  then (\ref{lw1}) can be seen as a signomial geometric program. See \cite[Section 9.1]{boyd}
for the definition of a signomial geometric program. Unfortunately, signomial geometric programs are non-convex and
are typically much harder to solve than geometric programs.}
Note also that $f(0)-\rho$ may be strictly smaller than $s(f,\mathbf{g})$.
\begin{exm}\label{example1}
Suppose $n=2$, $m=1$, $f= x^2-2xy+y^2$, $g_1 = x+y$. Then $G(0)_{\operatorname{gp}} = f_{\operatorname{gp}} = 0$,
$G(\lambda)_{\operatorname{gp}} = -\infty$ for $\lambda>0$, so $s(f,\mathbf{g})= G(0)_{\operatorname{gp}} =0$.
In this example, $f(0)-\rho = -\infty$. Similarly, if $f=x+y+x^2-2xy+y^2$, $g_1= x+y$, then $G(1)_{\operatorname{gp}} = 0$,
$G(\lambda) = -\infty$ for $\lambda\ge 0$, $\lambda \ne 1$, $s(f,\mathbf{g})= G(1)_{\operatorname{gp}} =0$, and
$f(0)-\rho = -\infty$.
\end{exm}
\section{Relaxation to a geometric program}
We discuss relaxations of (\ref{lw1}) which are geometric programs.
We consider a linear change of variables
\[
	\lambda_j = \sum_{k=0}^m a_{jk}\mu_k, \ j=0,\dots,m,
\]
where $\mu_0 :=1$ and $a_{jk}$, $j,k =0,\dots,m$ are real constants such that
\[
	a_{0k} = \begin{cases} 1 \text{ if } k=0 \\ 0 \text{ if } k \ne 0. \end{cases}
\]
Let $h_k := \sum_{j=0}^m a_{jk}g_j$, $k=0,\dots,m$, $H(\mu) := -\sum_{j=0}^m \mu_jh_j$. Clearly $G(\lambda) = H(\mu)$.
For $\alpha \in \Delta$, decompose $H(\mu)_{\alpha}$ as $H(\mu)_{\alpha} =H(\mu)_{\alpha}^+- H(\mu)_{\alpha}^-$, where
\[
	H(\mu)_{\alpha}^+ := -\sum_{(h_j)_{\alpha}<0} (h_j)_{\alpha}\mu_j, \ H(\mu)_{\alpha}^- := \sum_{(h_j)_{\alpha}>0} (h_j)_{\alpha}\mu_j.
\]
We take advantage of the inequality $\max\{a,b\} \ge |a-b|$, which holds for any nonnegative real numbers $a,b$.
Note that $\max\{a,b\} = |a-b|$  if and only if one of $a,b$ is zero. Define $h_j(0)^+ := \max\{ h_j(0),0\}$.
We consider the following program:
\begin{equation}
\label{lw2}
\left\{
\begin{array}{llr}
	\textrm{Minimize} & \multicolumn{2}{c}{\sum\limits_{j=1}^m \mu_jh_j(0)^++ \sum\limits_{\alpha\in\Delta^{<d}}(d-\vert\alpha\vert)
	 \left[\left(\frac{w_\alpha}{d}\right)^{d}\,\left(\frac{\alpha}{\mathbf{z}_\alpha}\right)^\alpha
	\right]^{1/(d-\vert\alpha\vert)}}\\
	\textrm{\mbox{s.t.}} & \sum\limits_{\alpha\in\Delta}z_{\alpha,i} \le H(\mu)_{d,i}, & i=1,\ldots,n\\
	 & \left(\frac{\mathbf{z}_\alpha}{\alpha}\right)^\alpha \ge \left(\frac{w_\alpha}{d}\right)^{d}, & \alpha\in\Delta^{=d} \\
	 & w_{\alpha}\ge \max\{ H(\mu)_{\alpha}^+, H(\mu)_{\alpha}^-\}, & \alpha \in \Delta \\
	 & \sum\limits_{k=0}^m a_{jk}\mu_k \ge 0, & j=1,\dots,m
\end{array}\right.
\end{equation}
where, for every $\alpha\in\Delta$, the unknowns $\mathbf{z}_\alpha = (z_{\alpha,i})\in[0,\infty)^n$ satisfy $z_{\alpha,i}= 0$
if and only if $\alpha_i=0$, the unknowns $\mathbf{w} = (w_{\alpha})_{\alpha\in \Delta}$ satisfy $w_{\alpha}>0$, and the
unknowns $\mu = (\mu_1,\dots,\mu_m)$ satisfy $\mu_k> 0$.
\begin{thm} \label{change variables}
Assume that exactly one of $a_{j0},\dots,a_{jm}$ is strictly positive, or all of $a_{j0},\dots,a_{jm}$ are non-negative,
for each $j= 1,\dots,m$, and exactly one of $(h_0)_{d,i},\dots,(h_m)_{d,i}$ is strictly negative, for each $i=1,\dots,n$.
Then (\ref{lw2}) is a geometric program. Moreover, if $\rho$ denotes the optimum value of (\ref{lw2}), then
$f_{\operatorname{gp},\mathbf{g}}:=-h_0(0)-\rho$ is a lower bound for 
$s(f,\mathbf{g})$.
\end{thm}
\begin{proof}
The constraint $\sum_{\alpha\in \Delta} z_{\alpha,i} \le H(\mu)_{d,i}$ can be written in the form
$\sum_{\alpha\in \Delta} z_{\alpha,i} +\sum_{j\ne j_i} (h_j)_{d,i}\mu_j  \le -(h_{j_i})_{d,i}\mu_{j_i}$,
where $j_i$ is the unique $j$ such that $(h_j)_{d,i}<0$.  If exactly one of $a_{j0},\dots,a_{jm}$ is strictly positive,
then the constraint $\sum_{k=0}^m a_{jk}\mu_k \ge 0$ can be written in the form
$-\sum_{k\ne k_j} a_{jk}\mu_k \le a_{j{k_j}}\mu_{k_j}$ where $k_j$ is the unique $k$ such that $a_{jk}>0$.
If all of $a_{j0},\dots,a_{jm}$ are non-negative, then the constraint $\sum_{k=0}^m a_{jk}\mu_k \ge 0$ is the empty constraint.
Also, for each $\alpha \in \Delta$, $H(\mu)_{\alpha}^+$ and $H(\mu)_{\alpha}^-$ are posinomials in the $\mu_j$ and,
for each $j=1,\dots,m$, $h_j(0)^+\ge0$.  It follows from these facts that (\ref{lw2}) is a geometric program.
Suppose now that $(\mathbf{z},\mathbf{w},\mu)$ is a feasible point for (\ref{lw2}).
Let $\lambda_j= \sum_{k=0}^m a_{jk}\mu_k$, $j=0,\dots,m$. Then $\lambda_j\ge 0$, $j=1,\dots,m$, and $H(\mu)=G(\lambda)$.
Also, $w_{\alpha} \ge \max\{ H(\mu)_{\alpha}^+, H(\mu)_{\alpha}^-\} \ge |H(\mu)_{\alpha}| = |G(\lambda)_{\alpha}|$,
for each $\alpha \in \Delta$, so $(\mathbf{z},\lambda)$ is a feasible point of (\ref{lw1}). Also,
\[
\begin{array}{lcl}
	h_0(0) & + & \sum\limits_{j=1}^m \mu_jh_j(0)^++ \sum\limits_{\alpha\in\Delta^{<d}}(d-\vert\alpha\vert)
	 \left[\left(\frac{w_\alpha}{d}\right)^{d}\,\left(\frac{\alpha}{\mathbf{z}_\alpha}\right)^\alpha
	\right]^{1/(d-\vert\alpha\vert)}\\
	 & \ge & \sum\limits_{j=0}^m \mu_jh_j(0)+ \sum\limits_{\alpha\in\Delta^{<d}}(d-\vert\alpha\vert)
	 \left[\left(\frac{w_\alpha}{d}\right)^{d}\,\left(\frac{\alpha}{\mathbf{z}_\alpha}\right)^\alpha
	 \right]^{1/(d-\vert\alpha\vert)}\\
	 & \ge & \sum\limits_{j=0}^m \mu_jh_j(0)+ \sum\limits_{\alpha\in\Delta^{<d}}(d-\vert\alpha\vert)
	 \left[\left(\frac{H(\mu)_\alpha}{d}\right)^{d}\,\left(\frac{\alpha}{\mathbf{z}_\alpha}\right)^\alpha
	 \right]^{1/(d-\vert\alpha\vert)}\\
	 & = & \sum\limits_{j=0}^m \lambda_jg_j(0)+ \sum\limits_{\alpha\in\Delta^{<d}}(d-\vert\alpha\vert)
	 \left[\left(\frac{G(\lambda)_\alpha}{d}\right)^{d}\,\left(\frac{\alpha}{\mathbf{z}_\alpha}\right)^\alpha
	 \right]^{1/(d-\vert\alpha\vert)},
\end{array}
\]
so, by Theorem \ref{general},
\[
	-h_0(0)-\sum\limits_{j=1}^m \mu_jh_j(0)^+- \sum\limits_{\alpha\in\Delta^{<d}}(d-\vert\alpha\vert)
	 \left[\left(\frac{w_\alpha}{d}\right)^{d}\,\left(\frac{\alpha}{\mathbf{z}_\alpha}\right)^\alpha
	\right]^{1/(d-\vert\alpha\vert)}
\]
is a lower bound for $s(f,\mathbf{g})$. 
Since this is valid for any feasible point $(\mathbf{z},\mathbf{w},\mu)$ of (\ref{lw2}), it follows that $-h_0(0)-\rho$
is a lower bound for $s(f,\mathbf{g})$. 
\end{proof}
One would expect the bound $f_{\operatorname{gp},\mathbf{g}}$ given by Theorem \ref{change variables} to be  best when
$h_j(0)\ge 0$ for all $j=1,\dots,m$ and  one of $H(\mu)_{\alpha}^+$, $H(\mu)_{\alpha}^-$ is identically zero  for almost
all $\alpha \in \Delta$. 
Note that one of $H(\mu)_{\alpha}^+$, $H(\mu)_{\alpha}^-$ is zero if and only if
$\max\{ H(\mu)_{\alpha}^+, H(\mu)_{\alpha}^-\} = |H(\mu)|$ if and only if the $(h_0)_{\alpha}, \dots, (h_m)_{\alpha}$
are all greater than or equal to $0$ or all less than or equal to $0$.

The source code of a SAGE program, which outputs the lower bound $f_{\operatorname{gp},\mathbf{g}}$ of $f$ on $K_{\mathbf{g}}$
described in Theorem \ref{change variables} (or the statement ``not a geometric program'' if the hypothesis of
Theorem \ref{change variables} fails to hold), is available at the address
\href{https://github.com/mghasemi/CvxAlgGeo}{github.com/mghasemi/CvxAlgGeo}. The input is 
the polynomials $f,g_1,\dots,g_m$, and the matrix $A = (a_{jk})$.

It is important to understand that $f_{\operatorname{gp},\mathbf{g}}$ depends on the choice of $A$. We write
$f_{\operatorname{gp},\mathbf{g}} = f_{\operatorname{gp},\mathbf{g}}^A$ when we wish to emphasis this fact.
A deficiency in Theorem \ref{change variables} is that there is no indication of how the matrix $A$ should be chosen.
We proceed to address this deficiency now, in an important special case.
\begin{thm}\label{important special case}
Assume $(*)$: for each $i=1,\dots,n$ there exists $0\le j_i \le m$ such that $(g_{j_i})_{d,i}<0$ and $(g_j)_{d,i}=0$ for $j>j_i$.
Then (1) there exists a canonically defined lower triangular matrix $A = (a_{jk})$ such that $a_{jj}=1$, $a_{jk}=0$ if $k>j$
and $a_{jk}\le 0$ if $k<j$, such that the hypothesis of Theorem \ref{change variables} holds for $A$, and (2) if, in addition,
$\Delta(-g_j)=\emptyset$ for $j=1,\dots,m$, and $g_k(0)+\sum_{j>k}a_{jk}g_j(0)\ge 0$, for all $k=1,\dots,m$, then
$f_{\operatorname{gp},\mathbf{g}}^A = s(f,\mathbf{g})$.
\end{thm}
\begin{proof}
Assume that ($*$) holds. We know that  $h_k = \sum_{j=0}^m a_{jk}g_j$, $k=0,\dots,m$. Choose $a_{jj}=1$ and $a_{jk}=0$ if $k>j$.
Thus $h_k = g_k+\sum_{j>k} a_{jk}g_j$. For each  $k<j$ define $a_{jk}$ by induction on $j$, as follows. For each $k<j$ and
each $i$ such that $j=j_i$, 
choose $a_{jk} \le 0$ as large as possible in absolute value so that
$(h_k)_{d,i} = (g_k)_{d,i}+\sum_{j\ge j'>k} a_{j'k}(g_{j'})_{d,i}\ge 0$, i.e.,
%
\[
	a_{jk} := \min_{i=1,\dots,n}\{ \displaystyle-[(g_k)_{d,i}+\sum_{j>j'>k} a_{j'k}(g_{j'})_{d,i}]/(g_j)_{d,i},0 :  j=j_i\}.
\]
Note that $a_{j'k}$ is already defined, by induction on $j$, for $j>j'>k$. By choice of $a_{jk}$, for each $i$,
if $j=j_i$, 
then $(h_k)_{d,i} =0$ for $k>j$, $(h_j)_{d,i} = (g_j)_{d,i} <0$,
and $(h_k)_{d,i}\ge 0$ for $k<j$. It follows that the hypothesis of Theorem \ref{change variables} holds.

Assume now that $\Delta(-g_j)=\emptyset$ for $j=1,\dots,m$, and $g_k(0)+\sum_{j>k}a_{jk}g_j(0)\ge 0$, for all $k=1,\dots,m$.
We want to show $f_{\operatorname{gp},\mathbf{g}}^A = s(f,\mathbf{g})$. In view of Theorem \ref{change variables} it suffices to
show $f_{\operatorname{gp},\mathbf{g}}^A \ge s(f,\mathbf{g})$. By our hypothesis, $h_k(0) = g_k(0)+\sum_{j>k}a_{jk}g_j(0)\ge 0$,
for all $k=1,\dots,m$, and $(g_j)_{\alpha} =0$ for all $\alpha \in \Delta$ and all $j=1,\dots,m$. Consequently,
$H(\mu)_{\alpha} = G(\lambda)_{\alpha} = f_{\alpha}$
and 
$\max\{ H(\mu)_{\alpha}^+, H(\mu)_{\alpha}^-\} = |f_{\alpha}|$,  for all $\alpha \in \Delta$. Let $\lambda \in [0,\infty)^m$
and let $\epsilon>0$ be given and let $\mathbf{z}$ be a feasible point of the the program
\begin{equation}\label{lw'}
\left\{
\begin{array}{llr}
	\textrm{Minimize} & \multicolumn{2}{c}{\sum\limits_{\alpha\in\Delta(f)^{<d}}(d-\vert\alpha\vert)
	 \left[\left(\frac{f_\alpha}{d}\right)^{d}\,\left(\frac{\alpha}{\mathbf{z}_\alpha}\right)^\alpha
	\right]^{1/(d-\vert\alpha\vert)}}\\
	\textrm{\mbox{s.t.}} & \sum\limits_{\alpha\in\Delta(f)}z_{\alpha,i} \le G(\lambda)_{d,i}, & i=1,\ldots,n\\
	& \left(\frac{\mathbf{z}_\alpha}{\alpha}\right)^\alpha =\left(\frac{f_\alpha}{d}\right)^{d}, & \alpha\in\Delta(f)^{=d}
\end{array}\right.
\end{equation}
Choose $\lambda' \in [0,\infty)^m$ so that $\lambda_j'>\lambda_j$, $j=1,\dots,m$ and $G(\lambda)_{d,i}\le G(\lambda')_{d,i}$,
$i=1,\dots,n$ with $\lambda'$ so close to $\lambda$ that $|\sum_{j=1}^m (\lambda_j'-\lambda_j)g_j(0)|<\epsilon$.
Existence of $\lambda'$ is a consequence of ($*$). For each $i=1,\dots,n$ there exists $0\le k(=j_i) \le m$ such that
$(g_k)_{d,i}<0$ and $(g_j)_{d,i}=0$ for $j>k$, so
\[
\begin{array}{lcl}
	G(\lambda')_{d,i}-G(\lambda)_{d,i} & = & \displaystyle-\sum_{j=0}^m (\lambda_j'-\lambda_j)(g_j)_{d,i}\\
	 & = & \displaystyle-\sum_{j=0}^k (\lambda_j'-\lambda_j)(g_j)_{d,i}\\
	 & \ge & \displaystyle-(\lambda_k'-\lambda_k) (g_k)_{d,i}- \sum_{j=0}^{k-1}(\lambda_j'-\lambda_j)|(g_j)_{d,i}|.
\end{array}
\]
Thus we can choose $\lambda_1',\dots,\lambda_m'$ so that $\lambda_j'>\lambda_j$,
$j=1,\dots,m$, $|\sum_{j=1}^m (\lambda_j'-\lambda_j)g_j(0)|<\epsilon$, and so that if $k=j_i$,
then $(\lambda_j'-\lambda_j)|(g_j)_{d,i}| \le -\frac{1}{2^{k-j}}(\lambda_k'-\lambda_k)(g_k)_{d,i}$, for $j<k$,
so $-(\lambda_k'-\lambda_k)(g_k)_{d,i}\ge \sum_{j=0}^{k-1}(\lambda_j'-\lambda_j)|(g_j)_{d,i}|$.
Choose $\mu' \in [0,\infty)^m$ so that $\lambda_j' = \sum_{k\le j} a_{kj}\mu_k'$, $j=0,\dots,m$, $\lambda_0'=\mu_0'= 1$.
Using $a_{jk}\le 0$ for $j>k$ and $\lambda_j'>0$ one sees that $\mu_j'>0$, $j=1,\dots,m$.
Finally, let $\mathbf{w}_{\alpha} = |f_{\alpha}|$ for all $\alpha \in \Delta$. One checks that $(\mathbf{z},\mathbf{w},\mu')$
is a feasible point for program (\ref{lw2}) and
\[
\begin{array}{lcl}
	H(\mu')(0) & - & \sum\limits_{\alpha\in\Delta(f)^{<d}}(d-\vert\alpha\vert)\left[
	 \left(\frac{\mathbf{w}_\alpha}{d}\right)^{d}\,\left(\frac{\alpha}{\mathbf{z}_\alpha}\right)^\alpha
	\right]^{1/(d-\vert\alpha\vert)} \\
	 & = & G(\lambda')(0)-\sum\limits_{\alpha\in\Delta(f)^{<d}}(d-\vert\alpha\vert)\left[
	 \left(\frac{f_\alpha}{d}\right)^{d}\,\left(\frac{\alpha}{\mathbf{z}_\alpha}\right)^\alpha
	\right]^{1/(d-\vert\alpha\vert)}\\
	 & \ge & G(\lambda)(0)-\sum\limits_{\alpha\in\Delta(f)^{<d}}(d-\vert\alpha\vert)\left[
	 \left(\frac{f_\alpha}{d}\right)^{d}\,\left(\frac{\alpha}{\mathbf{z}_\alpha}\right)^\alpha
	\right]^{1/(d-\vert\alpha\vert)} - \epsilon.
\end{array}
\]
It follows that $f_{\operatorname{gp},\mathbf{g}}^A\ge G(\lambda)_{\operatorname{gp}}-\epsilon$.
\end{proof}

\begin{rem} \label{extensions and clarifications} \

(1) If ($*$) holds, $\Delta(-g_j)=\emptyset$ for $j=1,\dots,m$, and the matrix $A$ is chosen as in
Theorem \ref{important special case}, then $H(\mu)_{\alpha}=f_{\alpha}$, for all $\alpha \in \Delta$, and program
(\ref{lw2}) reduces to the following one:
\begin{equation}
\label{lw3}
\left\{\begin{array}{lll}
	\textrm{Minimize} & \multicolumn{2}{c}{\sum\limits_{j=1}^m \mu_jh_j(0)^++ \sum\limits_{\alpha\in\Delta^{<d}}(d-\vert\alpha\vert)\left[
	 \left(\frac{f_\alpha}{d}\right)^{d}\,\left(\frac{\alpha}{\mathbf{z}_\alpha}\right)^\alpha
	\right]^{1/(d-\vert\alpha\vert)}}\\
	\textrm{\mbox{s.t.}} & \sum\limits_{\alpha\in\Delta}z_{\alpha,i} \le H(\mu)_{d,i}, & i=1,\ldots,n\\
	& \left(\frac{\mathbf{z}_\alpha}{\alpha}\right)^\alpha = \left(\frac{f_\alpha}{d}\right)^{d}, & \alpha\in\Delta^{=d}\\
	& \sum\limits_{k=0}^m a_{jk}\mu_k \ge 0, & j=1,\dots,m
\end{array}\right.
\end{equation}
where, for every $\alpha\in\Delta$, the unknowns $\mathbf{z}_\alpha = (z_{\alpha,i})\in[0,\infty)^n$ satisfy $z_{\alpha,i}= 0$
if and only if $\alpha_i=0$, and the unknowns $\mu = (\mu_1,\dots,\mu_m)$ satisfy $\mu_k> 0$.

(2) If condition ($*$) is replaced by the stronger condition ($**$): for each $i=1,\dots,n$ there exists $k$ such that
$(g_k)_{d,i}<0$, $(g_j)_{d,i} =0$ for $j>k$ and $(g_j)_{d,i}\ge 0$ for $0<j<k$, then $a_{jk}=0$ for all $j>k>0$,
$h_k(0)=g_k(0)+\sum_{j>k}a_{jk}g_j(0) = g_k(0)$, and the condition $g_k(0)+\sum_{j>k}a_{jk}g_j(0)\ge 0$ reduces to
$g_k(0)\ge 0$, $k=1,\dots,m$.

(3) If $m=1$ and $g_1= M-(x_1^d+\dots+x_n^d)$, then parts (1) and (2) of Theorem \ref{important special case} apply,
yielding a lower bound for $f$ on the hyperellipsoid $B_M =\{ \mathbf{x} \in \mathbb{R}^n : x_1^{d}+\dots+x_n^{d}\le M\}$,
computable by geometric programming. Note that, in this example,
\[
	A =\begin{pmatrix} 1&0\\-c&1\end{pmatrix}, \text{ where } c = \max\{ 0, f_{d,i} : i=1,\dots,n\}.
\]
Note also that, in this example, the program (\ref{lw3}) for computing $f_{\rm{\operatorname{gp}},\mathbf{g}}^A$ is not
the same as the program in \cite[Theorem 2.4]{gha-lass-mar}, but it is equivalent, i.e., it produces the same output.
In fact, because fewer variables and constraints are involved, the program (\ref{lw3}) is faster than the one in
\cite{gha-lass-mar}. These facts were also verified experimentally, by redoing Examples 4.1--4.5 of \cite{gha-lass-mar}
using program (\ref{lw3}) in place of the program used in \cite{gha-lass-mar}.

(4) Fix $N_i>0$, $i=1,\dots,n$. If $m=n$, $g_i= N_i^d-x_i^d$, $i=1,\dots,n$, then parts (1) and (2) of Theorem
\ref{important special case} apply. This gives a lower bound for $f$ on the hypercube $\prod_{i=1}^n [-N_i,N_i]$.
Observe that, in this example, the stronger condition ($**$) holds, and
\[
	A = \begin{pmatrix} 1 & 0&\dots&0&0\\ -c_1 & 1&\dots&0&0\\&&\dots&&\\ -c_n&0&\dots&0&1\end{pmatrix} \text{ where } c_j = \max\{0,f_{d,j}\}, \ j=1,\dots,n
\]

(5) In computing a lower bound for $f$ on $\prod_{j=1}^n [-N_j,N_j]$ using the method described in (4) one can take $d$
to be any even integer $\ge 2\lceil \frac{\deg(f)}{2}\rceil$. We will show later, in Section 5, that if $d >\deg f$, more
generally, if $d\ge \deg f$ and $f_{d,i}\le 0$, $i=1,\dots,n$, then the lower bound obtained in this way coincides with the trivial lower bound;
see Theorem \ref{hypercube case}.
%

(6) Fix $N_i >0$, $i=1,\dots,n$ and a partition $I_1,\dots,I_m$ of $\{ 1,\dots,n\}$, i.e., $I_1,\dots,I_m$ are nonempty
pairwise disjoint subsets of $\{1,\dots,n\}$ whose union is all of $\{1,\dots,n\}$. Set $\mathbf{g} = (g_1,\dots,g_m)$ where
$g_j:= 1- \sum_{i\in I_j} (x_i/N_i)^d$, $j=1,\dots,m$.
In this situation, condition ($**$) holds and Theorem \ref{important special case} (1) and (2) apply to produce lower bound
for $f$ on $K_{\mathbf{g}}$, where $K_{\mathbf{g}}$ is the product of hyperellipsoids defined by
\[
	K_{\mathbf{g}} := \{ \mathbf{x} \in \mathbb{R}^n : \sum_{i\in I_j}(x_i/N_i)^d\le 1, j=1,\dots,m\}.
\]

(7) Examples (3) and (4) can be seen as special cases of (6): If each $I_j$ is singleton, (6) produces the lower bound for $f$
on $\prod_{i=1}^n [-N_i,N_i]$ described in (4). If there is just one $I_j$, i.e., $m=1$ and $I_1 = \{ 1,\dots,n\}$, and if
$N_i= \root d \of{M}$, $i=1,\dots,n$, then (6) produces to the lower bound for $f$ on $B_M$ described in (3).

(8) Table \ref{RunTimeTable} records average running time for computation of $s(f,\mathbf{g})$ for large examples (where computation of $f_{\operatorname{sos},\mathbf{g}}^{(d)}$ is not possible). Here $\textbf{g} = (g_1,\dots,g_m)$, $g_j = 1 - \sum_{i\in I_j} x_i^d$, $j=1,\dots,m$, where $\{I_1,\dots,I_m\}$ is a partition of $\{1,\dots,n\}$. The average is taken over 10 randomly chosen partitions $\{I_1,\dots,I_m\}$ and 
polynomials $f$, each $f$ having $t$ terms and $\deg(f)\le d$ with coefficients chosen from $[-10,10]$. See also \cite[Table 2]{gha-lass-mar} and \cite[Table 3]{gha-mar2}.  

%
\begin{table}[ht]{
\caption{Average runtime for computation of $s(f,\mathbf{g})$ (seconds) for various $n, d$ and $t$.}\label{RunTimeTable}
\centering
\begin{tabular}{|c|c|cccc|}
\hline
$n$ & $d\backslash t$ & 50 & 100 & 150 & 200 \\
\hline
\multirow{3}{*}{10} & 20 & 3.330 & 23.761 & 79.369 & 170.521 \\
				    & 40 & 5.730 & 43.594 & 159.282 & 421.497 \\
				    & 60 & 6.524 & 68.625 & 191.126 & 531.442 \\
\hline
\multirow{3}{*}{20} & 20 & 8.364 & 63.198 & 193.431 & 562.243 \\
				    & 40 & 16.353 & 149.137 & 473.805 & 1102.579 \\
				    & 60 & 31.774 & 304.065 & 782.967 & 1184.263 \\
\hline
\multirow{3}{*}{30} & 20 & 12.746 & 107.285 & 353.803 & 776.831 \\
				    & 40 & 46.592 & 310.228 & 753.356 & 1452.159 \\
				    & 60 & 58.838 & 539.738 & 1271.102 & 1134.887 \\
\hline
\multirow{3}{*}{40} & 20 & 15.995 & 148.827 & 423.117 & 989.318 \\
				    & 40 & 60.861 & 414.188 & 1493.461 & 1423.965 \\
				    & 60 & 95.384 & 784.039 & 1305.201 & 1093.932 \\
\hline

\end{tabular}}
\end{table}

 (9) Table \ref{runtimecmp} records average running time for computation of $s(f,\mathbf{g})$ (the top number) and $f^{(d)}_{\operatorname{sos},\mathbf{g}}$ (the bottom number) for small examples. Here $\textbf{g} = (g_1,\dots,g_m)$, $g_j = 1 - \sum_{i\in I_j} x_i^d$, $j=1,\dots,m$, where $\{I_1,\dots,I_m\}$ is a partition of $\{1,\dots,n\}$. The average is taken over 10 randomly chosen partitions $\{I_1,\dots,I_m\}$ and 
polynomials $f$, each $f$ having $t$ terms and $\deg(f)\le d$ with coefficients chosen from $[-10,10]$.\footnote{\textbf{Hardware and Software
specifications.} Processor: Intel\textregistered~ Core\texttrademark2 Duo CPU P8400 @ 2.26GHz, Memory: 3 GB, OS: \textsc{Ubuntu}
14.04-32 bit, \textsc{Sage}-6.0} See also \cite[Table 1]{gha-lass-mar} and \cite[Table 2]{gha-mar2}.

\begin{table}[ht]{\footnotesize 
\caption{Average runtime for computation of $s(f,\mathbf{g})$ and $f^{(d)}_{\operatorname{sos},\mathbf{g}}$
for various $n, d$ and $t$.}\label{runtimecmp}

\centering
\begin{tabular}{|c|c|cccccccccc|}
\hline
$n$ & $d\backslash t$ & 10 & 30 & 50 & 100 & 150 & 200 & 250 & 300 & 350 & 400 \\
\hline
\multirow{6}{*}{3} & \multirow{2}{*}{4} & 0.034 & 0.092 &  &  &  &  &  &  &  & \\
					&				   & 0.026 & 0.026 &  &  &  &  &  &  &  & \\
				  & \multirow{2}{*}{6} & 0.041 & 0.105 &  &  &  &  &  &  &  & \\
				    	&				   & 0.108 & 0.103 &  &  &  &  &  &  &  & \\
				  & \multirow{2}{*}{8} & 0.045 & 0.164 &  &  &  &  &  &  &  & \\
				  	&				   & 0.522 & 0.510 &  &  &  &  &  &  &  & \\
\hline
\multirow{6}{*}{4} & \multirow{2}{*}{4} & 0.042 & 0.116 &  &  &  &  &  &  &  & \\
					&				   & 0.045 & 0.054 &  &  &  &  &  &  &  & \\
				  & \multirow{2}{*}{6} & 0.043 & 0.133 & 0.285 &  &  &  &  &  &  & \\
				    	&				   & 0.662 & 0.632 & 0.624 &  &  &  &  &  &  & \\
				  & \multirow{2}{*}{8} & 0.053 & 0.191 & 0.446 & 2.479 & 7.089 &  &  &  &  & \\
				  	&				   & 12.045 & 13.019 & 11.956 & 12.091 & 12.307 &  &  &  &  & \\
\hline
\multirow{6}{*}{5} & \multirow{2}{*}{4} & 0.039 & 0.124 & 0.295 &  &  &  &  &  &  & \\
					&				   & 0.181 & 0.154 & 0.127 &  &  &  &  &  &  & \\
				  & \multirow{2}{*}{6} & 0.052 & 0.156 & 0.396 & 1.677 & 5.864 & 13.814 &  &  &  & \\
				    	&				   & 6.429 & 6.530 & 6.232 & 6.425 & 6.237 & 6.469 &  &  &  & \\
				  & \multirow{2}{*}{8} & 0.056 & 0.219 & 0.528 & 3.046 & 7.663 & 23.767 & 44.699 & 88.104 & 123.986 & 179.126\\
				  	&				   & 340.321 & 243.860 & 225.746 & 205.621 & 222.619 & 220.887 & 224.018 & 218.994 & 219.085 & 213.348\\
\hline
\multirow{6}{*}{6} & \multirow{2}{*}{4} & 0.043 & 0.140 & 0.340 & 1.545 &  &  &  &  &  & \\
					&				   & 0.453 & 0.422 & 0.422 & 0.448 &  &  &  &  &  & \\
				  & \multirow{2}{*}{6} & 0.053 & 0.194 & 0.429 & 2.079 & 7.138 & 16.212 & 34.464 & 71.465 & 102.342 & 166.345 \\
				    	&				   & 48.239 & 47.752 & 47.776 & 51.268 & 51.008 & 48.012 & 49.908 & 53.135 & 51.542 & 52.874 \\
				  & \multirow{2}{*}{8} & 0.066 & 0.251 & 0.681 & 3.389 & 11.985 & 36.495 & 74.088 & 148.113 & 174.163 & 269.544 \\
				    	&				   & -- & -- & -- & -- & -- & -- & -- & -- & -- & -- \\
\hline
\end{tabular}}
\end{table}

(10) Table \ref{PartitionTable} computes average values for the relative error
\[
	R=  \frac{-s(-f,\mathbf{g})-s(f,\mathbf{g})}{-(-f)_{\operatorname{sos},\mathbf{g}}^{(d)}-f_{\operatorname{sos},\mathbf{g}}^{(d)}}
\]
for small examples. Here $\textbf{g} = (g_1,\dots,g_m)$, $g_j = 1 - \sum_{i\in I_j} x_i^d$, $j=1,\dots,m$, where $\{I_1,\dots,I_m\}$ is a  partition of $\{1,\dots,n\}$. The average is taken over 20 randomly chosen partitions $\{I_1,\dots,I_m\}$ and  polynomials $f$, each $f$ having $t$ terms and $\deg(f)\le d$ with coefficients chosen from $[-10,10]$. Table \ref{PartitionTable} would seem to confirm that for fixed $n,d$
the quality of the bound $s(f,\mathbf{g})$ is best when $t$ is small, and for fixed $d,t$ the quality of the bound $s(f,\mathbf{g})$
is best when $n$ is large.  Comparison of Table \ref{PartitionTable} with \cite[Table 3]{gha-lass-mar} would seem to indicate
that the quality of the bound $s(f,\mathbf{g})$ is best when $m=1$.
%
\begin{table}[ht]{\footnotesize
\caption{Average values of $R$ for various $n, d$ and $t$.}\label{PartitionTable}
\centering
\begin{tabular}{|c|c|ccccccccc|}
\hline
$n$ & $d\backslash t$ & 5 & 10 & 50 & 100 & 150 & 200 & 250 & 300 & 400 \\
\hline
\multirow{3}{*}{3} & 4 & 1.3688 & 1.6630 & & & & & & & \\
				   & 6 & 1.5883 & 2.0500 & 4.3726 & & & & & & \\
				   & 8 & 2.0848 & 2.7636 & 5.6391 & 5.8140 & 6.9135 & & & & \\
\hline
\multirow{3}{*}{4} & 4 & 1.2183 & 1.3420 & 3.3995 & & & & & & \\
				   & 6 & 1.3816 & 2.9584 & 3.1116 & 4.6891 & 5.8067 & 6.5150 & & & \\
				   & 8 & 1.7630 & 2.2038 & 3.2685 & 4.4219 & 5.7929 & 7.0841 & 7.9489 & 8.7924 & 9.6068 \\
\hline
\multirow{3}{*}{5} & 4 & 1.2566 & 1.6701 & 3.1867 & 4.6035 & & & & & \\
				   & 6 & 2.3807 & 2.9424 & 3.3077 & 4.7939 & 5.6523 & 7.1996 & 8.6194 & 9.3317 & 10.134\\
				   & 8 & 1.5557 & 1.9754 & 2.5204 & 3.9815 & 4.5404 & 5.1756 & 6.2214 & 6.6919 & 7.8921 \\
\hline
\multirow{3}{*}{6} & 4 & 1.2069 & 1.3876 & 3.0639 & 4.5326 & 4.9645 & 6.1414 & & & \\
				   & 6 & 1.2602 & 1.4854 & 2.8236 & 4.0256 & 4.5797 & 6.3479 & 6.9487 & 7.4866 & 8.8435 \\
				   & 8 & 1.0478 & 1.1616 & 2.4884 & 3.3896 & 4.0870 & 5.0809 & 5.6932 & 6.1994 & 10.567 \\
\hline
\multirow{2}{*}{7} & 4 & 1.1943 & 1.3360 & 2.8592 & 4.5334 & 5.5064 & 5.9837 & 7.4782 & 7.3568 & \\
				   & 6 & 1.2604 & 1.4529 & 2.6962 & 3.8035 & 4.5351 & 6.0305 & 6.1478 & 6.6746 & 9.5049 \\
\hline
\multirow{2}{*}{8} & 4 & 1.1699 & 1.4274 & 2.5942 & 4.0914 & 5.4111 & 6.2111 & 7.4479 & 8.3168 & 9.1369 \\
				   & 6 & 1.0454 & 1.1270 & 2.1316 & 3.2807 & 3.5468 & 4.8955 & 5.1211 & 5.4214 & 7.1872 \\
\hline
\multirow{1}{*}{9} & 4 & 1.2158 & 1.3214 & 2.9305 & 4.0624 & 5.9063 & 6.5985 & 7.9552 & 8.4233 & 11.810 \\
\hline
\multirow{1}{*}{10} & 4 & 1.1476 & 1.3441 & 2.2393 & 3.7136 & 5.7739 & 6.1791 & 6.3211 & 8.7773 & 10.428 \\
\hline
\end{tabular}}
\end{table}
\end{rem}

\section{Variants of Theorem \ref{important special case}}

We explain how the hypothesis of Theorem \ref{important special case} can be weakened a bit when $m=1$.
\begin{thm}\label{m=1}
Suppose $m=1$ and $(g_1)_{d,i}=0$ $\Rightarrow$ $f_{d,i}>0$ for $i=1,\dots,n$. Choose $c$ so that
$(g_1)_{d,i}<0$ $\Rightarrow$ $c\ge -\frac{f_{d,i}}{(g_1)_{d,i}}$ and  $(g_1)_{d,i}>0$ $\Rightarrow$ $c> -\frac{f_{d,i}}{(g_1)_{d,i}}$,
for each $i=1,\dots,n$. Then (1) the hypothesis of Theorem \ref{change variables} holds for $A:= \begin{pmatrix} 1&0\\-c&1\end{pmatrix},$
and (2) if, in addition, $\Delta(-g_1)=\emptyset$, $g_1(0)\ge 0$, and $c>0$, then $f_{\operatorname{gp},\mathbf{g}}^A = s(f,\mathbf{g})$.
\end{thm}
\begin{proof}
Since $A = \begin{pmatrix} 1&0\\-c&1\end{pmatrix}$, $\lambda_0=\mu_0 = 1$, $\lambda_1=-c+\mu_1$, and
$G(\lambda) = f-\lambda_1g_1= f-(-c+\mu_1)g_1 = (f+cg_1)-\mu_1g_1$, so $H(\mu)= -h_0-\mu_1 h_1$ where $h_0 = -(f+cg_1)$
and $h_1 = g_1$. The hypothesis of Theorem \ref{change variables} is that exactly one of $(h_0)_{d,i}, (h_1)_{d,i}$ is
strictly negative for each $i=1,\dots,n$. Since $(h_0)_{d,i} = -(f_{d,i}+c(g_1)_{d,i})$ and $(h_1)_{d,i}=(g_1)_{d,i}$ the
proof of (1) is completely straightforward. The proof of (2) is similar to the proof of Theorem \ref{important special case}
(2), but it is a good deal simpler: If $\lambda_1 \in [0,\infty)$ and $\mathbf{z}$ is a feasible point of (\ref{lw'}) then
$(\mathbf{z},\mathbf{w},\mu_1)$, where $w_{\alpha}=|f_{\alpha}|$ and $\mu_1 = c+\lambda_1$, is a feasible point of (\ref{lw2}),
and
\begin{align*}
&H(\mu_1)(0)- \sum\limits_{\alpha\in\Delta(f)^{<d}}(d-\vert\alpha\vert)\left[
\left(\frac{\mathbf{w}_\alpha}{d}\right)^{d}\,\left(\frac{\alpha}{\mathbf{z}_\alpha}\right)^\alpha
\right]^{1/(d-\vert\alpha\vert)} \\
&= G(\lambda_1)(0)-\sum\limits_{\alpha\in\Delta(f)^{<d}}(d-\vert\alpha\vert)\left[
\left(\frac{f_\alpha}{d}\right)^{d}\,\left(\frac{\alpha}{\mathbf{z}_\alpha}\right)^\alpha
\right]^{1/(d-\vert\alpha\vert)}.
\end{align*}
It follows that $f_{\operatorname{gp},\mathbf{g}}^A\ge G(\lambda)_{\operatorname{gp}}$.
\end{proof}
Similarly, the hypothesis of Theorem \ref{important special case} can be weakened when $m=2$ and $f_{d,i}=0$ for $i=1,\dots,n$.
\begin{thm}\label{m=2}
Suppose $m=2$, $f_{d,i}=0$ for $i=1,\dots,n$ and $(g_2)_{d,i}=0$ $\Rightarrow$ $(g_1)_{d,i} <0$ for $i=1,\dots,n$.
Choose $c$ so that $(g_2)_{d,i}<0$ $\Rightarrow$ $c\ge \frac{(g_1)_{d,i}}{(g_2)_{d,i}}$ and  $(g_2)_{d,i}>0$ $\Rightarrow$
$c> \frac{(g_1)_{d,i}}{(g_2)_{d,i}}$, for each $i=1,\dots,n$. Then (1) the hypothesis of Theorem \ref{change variables}
holds for $A:= \begin{pmatrix} 1&0&0\\0&1&0\\0&-c&1\end{pmatrix},$ and (2) if, in addition, $\Delta(-g_j)=\emptyset$ for
$j=1,2$, $g_k(0)\ge 0$, for $k=1,2$, and $c>0$, then
$f_{\operatorname{gp},\mathbf{g}}^A \ge \sup\{ G(\lambda)_{\operatorname{gp}} : \lambda_1>0, \lambda_2\ge 0\}$.
\end{thm}
\begin{proof}
Similar to the proof of Theorem \ref{m=1}.
\end{proof}
We also mention another variant of Theorem \ref{important special case}.
\begin{thm}\label{A=I}
Assume ($\dagger$): for each $i=1,\dots,n$, exactly one of the coefficients $(g_0)_{d,i},\dots,(g_m)_{d,i}$ is
strictly negative. Then (1) Theorem \ref{change variables} applies with $A := I$ (the identity matrix), (2)
$f_{\operatorname{gp},\mathbf{g}}^I \le s_0(f,\mathbf{g})$ and (3) if, in addition, $\Delta(-g_j)\cap \Delta(-g_k)=\emptyset$
for $0\le j<k\le m$ and $g_k(0)\ge 0$ for $k=1,\dots,m$ then $f_{\operatorname{gp},\mathbf{g}}^I = s_0(f,\mathbf{g})$.
Here, $s_0(f,\mathbf{g}) : = \sup\{ G(\lambda)_{\operatorname{gp}} : \lambda \in (0,\infty)^m\}$.
\end{thm}
\begin{proof}
(1) is clear. Arguing as in the proof of Theorem \ref{change variables} we see that $f_{\operatorname{gp},\mathbf{g}}^I \le f(0)-\rho$, where $\rho$ is the optimum value of program (\ref{lw1}) in Section 3, but where
the unknowns $\lambda = (\lambda_1,\dots,\lambda_m)$ are now required to satisfy the strict inequality $\lambda_j> 0$.  (2) follows from this by a rather obvious modification of the proof of Theorem \ref{general}. The extra hypothesis in (3) implies that $f_{\operatorname{gp},\mathbf{g}}^I = f(0)-\rho = s_0(f,\mathbf{g})$.
\end{proof}
Observe that when $m=1$ Theorem \ref{m=1} provides a generalization of both  Theorem \ref{important special case}
and Theorem \ref{A=I}. Ditto for Theorem \ref{m=2} when $m=2$ and $f_{d,i}=0$ for $i=1,\dots,n$.

Observe also that, as was the case with Theorem \ref{important special case}, if $m=1$, the hypothesis of
Theorem \ref{m=1} holds, $\Delta(-g_1)=\emptyset$, and  $A$ is chosen as in Theorem \ref{m=1}, or, if $m=2$, the hypothesis of
Theorem \ref{m=2} holds,  $\Delta(-g_j) = \emptyset$, $j=1,2$, and $A$ is chosen as in Theorem \ref{m=2}, then program (\ref{lw2})
reduces to program (\ref{lw3}). Similarly, if the hypothesis of Theorem \ref{A=I} holds, $\Delta(-g_j)\cap \Delta(-g_k)=\emptyset$
for $0\le j< k\le m$, and $A=I$, then $G(\lambda)_{\alpha}=\lambda_j(g_j)_{\alpha}$ for some $j$ (depending on $\alpha$), and
program (\ref{lw2}) reduces to the following one:
\begin{equation}
\label{lw4}
\left\{\begin{array}{llr}
	\textrm{Minimize} & \multicolumn{2}{c}{\sum\limits_{j=1}^m \lambda_jg_j(0)^++ \sum\limits_{\alpha\in\Delta^{<d}}(d-\vert\alpha\vert)
	 \left[\left(\frac{G(\lambda)_\alpha}{d}\right)^{d}\,\left(\frac{\alpha}{\mathbf{z}_\alpha}\right)^\alpha
	\right]^{1/(d-\vert\alpha\vert)}}\\
	\textrm{\mbox{s.t.}} & \sum\limits_{\alpha\in\Delta}z_{\alpha,i} \le G(\lambda)_{d,i}, & i=1,\ldots,n\\
	 & \left(\frac{\mathbf{z}_\alpha}{\alpha}\right)^\alpha = \left(\frac{G(\lambda)_\alpha}{d}\right)^{d}, & \alpha\in\Delta^{=d}
\end{array}\right.
\end{equation}
where, for every $\alpha\in\Delta$, the unknowns $\mathbf{z}_\alpha = (z_{\alpha,i})\in[0,\infty)^n$ satisfy
$z_{\alpha,i}= 0$ if and only if $\alpha_i=0$, and the unknowns $\lambda = (\lambda_1,\dots,\lambda_m)$ satisfy $\lambda_j> 0$.
%
\begin{rem}\label{applications}
Theorems \ref{important special case}, \ref{m=1}, \ref{m=2} and \ref{A=I} can be applied in various cases:

(1) Assume there exist $j_1,\dots,j_{\ell} \in \{1,\dots,m\}$ such that for each $i=1,\dots,n$ there exists $k$ such that
$(g_{j_k})_{d,i}<0$ and $(g_{j_p})_{d,i}=0$ for $k<p\le \ell$. Here, $j_0 := 0$.
Then one can apply Theorem \ref{important special case} to compute a lower bound for $f$ on $K_{(g_{j_1},\dots, g_{j_{\ell}})}$.
Since $K_{\mathbf{g}} \subseteq K_{( g_{j_1},\dots,g_{j_{\ell}})}$, this is also a lower bound for $f$ on $K_{\mathbf{g}}$.

(2) Assume there exist $j_1,\dots,j_{\ell} \in \{1,\dots,m\}$ such that for each $i=1,\dots,n$ exactly one of the coefficients
$-f_{d,i},(g_{j_1})_{d,i},\dots,(g_{j_{\ell}})_{d,i}$ is strictly negative. Then one can apply Theorem \ref{A=I} to compute a
lower bound for $f$ on $K_{(g_{j_1},\dots, g_{j_{\ell}})}$. Since $K_{\mathbf{g}} \subseteq K_{( g_{j_1},\dots,g_{j_{\ell}})}$,
this is also a lower bound for $f$ on $K_{\mathbf{g}}$.

(3) If there exists $j\in \{ 1,\dots,m\}$ such that $(g_j)_{d,i}=0$ $\Rightarrow$ $f_{d,i}>0$ for all $i=1,\dots,n$
then one can apply Theorem \ref{m=1} to compute a lower bound for $f$ on $K_{(g_j)}$. Since $K_{\mathbf{g}} \subseteq K_{( g_j)}$,
this is also a lower bound for $f$ on $K_{\mathbf{g}}$.

(4) If $f_{d,i}=0$, $i=1,\dots,n$ and there exists $j,k \in \{ 1,\dots,m\}$, $j\ne k$ such that
$(g_j)_{d,i}=0$ $\Rightarrow$ $(g_k)_{d,i}<0$ for all $i=1,\dots,n$ then one can apply Theorem \ref{m=2} to compute a lower bound
for $f$ on $K_{(g_j, g_k)}$. Since $K_{\mathbf{g}} \subseteq K_{( g_j, g_k)}$, this is also a lower bound for $f$ on $K_{\mathbf{g}}$.
\end{rem}
\begin{exm}\label{examples} \

(1) Suppose $n=3$, $m=2$, $d=2$, $f= p+qx+ry+sz$, $g_1=1-x^2-y^2$, $g_2=1-z^2$. In this example, $K_{\mathbf{g}}$ is a cylinder
and the lower bound for $f$ on $K_{\mathbf{g}}$ obtained using Theorem \ref{important special case} or Theorem \ref{A=I} is $p-\sqrt{q^2+r^2}-|s|$,
which is the exact minimum of $f$ on $K_{\mathbf{g}}$.\footnotemark\footnotetext{In examples (1), (2) and (3) the geometric program is so small that it can be solved by hand.}

(2) Suppose $n=3$, $m=2$, $d=2$, $f= p+qx+ry+sz$, $g_1= 2-x^2-y^2-z^2$, $g_2=1-z^2$. In this example, $K_{\mathbf{g}}$ is a
sphere with polar caps removed and the lower bound for $f$ on $K_{\mathbf{g}}$ obtained using Theorem \ref{important special case} or Theorem \ref{A=I} is
\[
	\begin{cases} p-\sqrt{q^2+r^2}-|s| &\text{ if } s^2\ge q^2+r^2 \\ p-\sqrt{2}\sqrt{q^2+r^2+s^2} &\text{ if } s^2\le q^2+r^2 \end{cases},
\]
which is the exact minimum of $f$ on $K_{\mathbf{g}}$.

(3) Suppose $n=2$, $m=2$, $d=2$, $f= p+qx+ry$, $g_1= 1-2x^2+y^2$, $g_2=1+x^2-y^2$. In this example, Theorem \ref{important special case}
does not apply, and the lower bound for $f$ on $K_{\mathbf{g}}$ obtained using Theorem \ref{A=I} is
$p-|q|\sqrt{2}-|r|\sqrt{3}$, which is the exact minimum of $f$ on $K_{\mathbf{g}}$.

(4) Suppose $n=2$, $m=1$, $d =4$.

i. If $f= 5x+6y+x^3-y^2$, $g_1=8-xy-x^4-y^4$, the lower bound obtained using Theorem \ref{important special case} is $-22.334$,
$f_{\operatorname{sos},\mathbf{g}}^{(d)}=-18.778$.

ii. If $f= 5x+6y+x^3-y^2+2xy$, $g_1=8-x^4-y^4+x^2y^2$, the lower bound obtained using Theorem \ref{important special case} is
$-31.815$, $f_{\operatorname{sos},\mathbf{g}}^{(d)}=-20.588$.

iii. If $f= 5x+6y+x^3-y^2+2xy$, $g_1=8+xy-x^4-y^4+x^2y^2$, the lower bound obtained using Theorem \ref{important special case}
is $-31.815$, $f_{\operatorname{sos},\mathbf{g}}^{(d)}=-23.247$.

In each of i, ii and iii, $f_{\operatorname{sos},\mathbf{g}}^{(d)}$ is the exact minimum of $f$ on $K_{\mathbf{g}}$.

(5) Suppose $n=3$, $m=2$, $d=4$, $f= x^4+y^4+z^4-y^3+xy$, $g_1=10x^3z+xyz^2+z^2-1$, $g_2=z^4-x^2yz$. In this example,
Theorem \ref{important special case} and Theorem \ref{m=2} do not apply, the lower bound obtained using Theorem \ref{A=I} is
$f_{\operatorname{gp},\mathbf{g}}^I=-0.485$, and $f_{\operatorname{sos},\mathbf{g}}^{(d)}=-0.468$. In this example,
$f_{\operatorname{gp},\mathbf{g}}^I=f_{\operatorname{gp}} =-0.485$ is the exact minimum of $f$ on $\mathbb{R}^3$ and
$f_{\operatorname{sos},\mathbf{g}}^{(d)}=-0.468$ is the exact minimum of $f$ on $K_{\mathbf{g}}$.

(6) Suppose $n=2$, $m=2$, $d=4$.

i. If $f=x+y$, $g_1=1-2y+6x^2-x^4$, $g_2=-x^3-y^4$, the lower bound obtained using Theorem \ref{important special case}
is $-4.64574$, which is the exact minumum of $f$ on $K_{\mathbf{g}}$.

ii. If $f=7y-2x^3$, $g_1=y+8y^2+2xy^2-x^4$, $g_2= -x^2y-y^4$, the lower bound obtained using Theorem \ref{important special case}
is $-88.3437$ and $f_{\operatorname{sos},\mathbf{g}}^{(d)}=-86.1157$, which is the exact minimum of $f$ on $K_{\mathbf{g}}$.

(7) Suppose $n=3$, $m=1$, $d=6$.

i. If $f=x+z^3+y^6+z^6$, $g_1=1-x^6+y^6$, then the lower bound obtained using Theorem \ref{m=1} or Theorem \ref{A=I} is $-1.25$,
which is equal to the exact minimum of $f$ on $K_{\mathbf{g}}$.

ii. If $f=x+z^3+x^6+y^6+z^6$, $g_1=1-x^6+y^6$, then Theorem \ref{m=1} applies with $A=\begin{pmatrix} 1&0\\-1&1\end{pmatrix}$,
and $f_{\operatorname{gp},\mathbf{g}}^A = f_{*,\mathbf{g}} = -1.25$.

(8) Suppose $n=2$, $m=2$, $d=6$, $f=-y-2x^2$, $g_1= y-x^4y+y^5-x^6-y^6$, $g_2= y-5x^2+x^4y-x^6-y^6$. The lower bound for $f$
obtained using Theorem \ref{important special case} is $-3.593$. Applying Theorem \ref{change variables} with
\[
	A = \begin{pmatrix}1&0&0\\0&1&1\\0&-1&1\end{pmatrix}
\]
yields the better lower bound $-2.652$. The exact minimum of $f$ on $K_{\mathbf{g}}$ is $-1.0494$.
\end{exm}
\section{The trivial bound on $\prod_{i=1}^n[-N_i,N_i]$}
Fix $f\in \mathbb{R}[\mathbf{x}]$ and $\mathbf{N} = (N_1,\dots,N_n)$, $N_i>0$, $i=1,\dots,n$. If $\mathbf{x} \in \prod_{i=1}^n [-N_i,N_i]$, then
\[
	f(\mathbf{x}) = \sum f_{\alpha}\mathbf{x}^{\alpha} \ge f(\mathbf{0})-\sum_{\alpha \in \Delta'(f)} |f_{\alpha}|\cdot |\mathbf{x}^{\alpha}| \ge f(\mathbf{0})-\sum_{\alpha\in \Delta'(f)} |f_{\alpha}|\cdot\mathbf{N}^{\alpha},
\]
where $\Delta'(f):= \{ \alpha \in \mathbb{N}^n : |\alpha|>0 \text{ and } f_{\alpha}\mathbf{x}^{\alpha} \text{ is not a square in } \mathbb{R}[\mathbf{x}]\}$
and $\mathbf{N}^{\alpha} := \prod_{i=1}^n N_i^{\alpha_i}$. Set
\begin{equation}
f_{\operatorname{tr},\mathbf{N}} := f(\mathbf{0}) - \sum_{\alpha\in \Delta'(f)} |f_{\alpha}|\cdot \mathbf{N}^{\alpha}.
\end{equation}
Thus $f_{\operatorname{tr},\mathbf{N}}$ is a lower bound for $f$ on the hypercube $\prod_{i=1}^n [-N_i,N_i]$.
We refer to $f_{\operatorname{tr},\mathbf{N}}$ as the \textit{trivial bound} for $f$ on $\prod_{i=1}^n [-N_i,N_i]$.
If $N_i = \root d \of{M}$, $i=1,\dots,n$, this coincides with the trivial bound defined in \cite[Section 3]{gha-lass-mar}.

Suppose now that $d$ is an even integer, $d \ge \max\{ 2, \deg f\}$. Define $\mathbf{g} = (N_1^d-x_1^d,\dots,N_n^d-x_n^d)$.
We want to compare $s(f,\mathbf{g})$ with $f_{\operatorname{tr},\mathbf{N}}$.
\begin{thm} \label{hypercube case} Set-up as above. Then  \

(1) $s(f,\mathbf{g}) \ge f_{\operatorname{tr},\mathbf{N}}$.

(2) If $f_{d,i}\le 0$ for $i=1,\dots,n$ then $s(f,\mathbf{g}) = f_{\operatorname{tr},\mathbf{N}}$.
In particular, if $\deg f <d$ then  $s(f,\mathbf{g}) = f_{\operatorname{tr},\mathbf{N}}$.
\end{thm}
We remark that the hypothesis of Theorem \ref{hypercube case} (2) is indeed necessary: If $n=1$, $f=x^2-x$, $d=2$, $N_1=1$,
then $f_{\operatorname{tr},\mathbf{N}} = -1$, $s(f,\mathbf{g})=f_{\operatorname{\operatorname{gp}}}=-\frac{1}{4}$.
\begin{proof}
By making the change of variables $y_i = \frac{x_i}{N_i}$, $i=1,\dots,n$, we are reduced to the case where $N_1=\dots = N_n=1$.
By definition of $\mathbf{g}$,
\[
G(\lambda) = f-\sum_{i=1}^n \lambda_i(1-x_i^d) = f(\mathbf{0})-\sum_{i=1}^n \lambda_i+\sum_{\alpha\in \Omega(f)}f_{\alpha}\mathbf{x}^{\alpha}+\sum_{i=1}^n (f_{d,i}+\lambda_i)x_i^d,
\]
and $s(f,\mathbf{g})$ is obtained by maximizing the objective function
\begin{equation}\label{obj}
f(\mathbf{0})-\sum_{i=1}^n \lambda_i- \sum_{\alpha\in\Delta(f)^{<d}} (d-\vert\alpha\vert)\left[
\left(\frac{f_\alpha}{d}\right)^{d}\,\left(\frac{\alpha}{\mathbf{z}_\alpha}\right)^\alpha
\right]^{1/(d-\vert\alpha\vert)}
\end{equation}
subject to
\begin{equation}\label{lw5}
\left\{
\begin{array}{lr}
	\sum\limits_{\alpha\in\Delta(f)}z_{\alpha,i} \ \le \ f_{d,i}+\lambda_i, & i=1,\ldots,n\\
	 \left(\frac{f_\alpha}{d}\right)^{d}\left(\frac{\alpha}{\mathbf{z}_\alpha}\right)^\alpha  =  1, & \alpha\in\Delta(f)^{=d}
\end{array}\right.
\end{equation}
where, $\lambda_i\ge 0$, $z_{\alpha,i}\ge 0$ and $z_{\alpha,i}= 0$ if and only if $\alpha_i=0$.

(1) Define $z_{\alpha,i}:= \alpha_i\frac{|f_{\alpha}|}{d}$ and
$\lambda_i := \max\{ 0, \sum_{\alpha\in \Delta(f)} z_{\alpha,i} - f_{d,i}\}$. One checks that, for this choice of
$z_{\alpha,i}$ and $\lambda_i$, the constraints of (\ref{lw5}) are satisfied. Observe also that $f_{d,i}\ge 0$ $\Rightarrow$
$\lambda_i \le \sum_{\alpha\in \Delta(f)} z_{\alpha,i}$ and $f_{d,i}<0$ $\Rightarrow$
$\lambda_i = \sum_{\alpha\in \Delta(f)}z_{\alpha,i}-f_{d,i}$. Consequently,
\begin{align*}
\sum_{i=1}^n \lambda_i \le& \sum_{f_{d,i}\ge 0} (\sum_{\alpha\in \Delta(f)} z_{\alpha,i})+ \sum_{f_{d,i}<0}(\sum_{\alpha\in\Delta(f)}z_{\alpha,i}-f_{d,i})\\
=& \sum_{i=1}^n (\sum_{\alpha\in \Delta(f)} z_{\alpha,i})+\sum_{f_{d,i}<0} |f_{d,i}|\\
=& \sum_{\alpha\in \Delta(f)}(\sum_{i=1}^n z_{\alpha,i})+\sum_{f_{d,i}<0} |f_{d,i}| \\
=& \sum_{\alpha\in \Delta(f)}|\alpha|\cdot \frac{|f_{\alpha}|}{d} +\sum_{f_{d,i}<0} |f_{d,i}|,
\end{align*}
and
\begin{align*}
s(f,\mathbf{g}) \ge& f(\mathbf{0})-\sum_{i=1}^n \lambda_i- \sum_{\alpha\in\Delta(f)^{<d}} (d-\vert\alpha\vert)\left[
\left(\frac{f_\alpha}{d}\right)^{d}\,\left(\frac{\alpha}{\mathbf{z}_\alpha}\right)^\alpha
\right]^{1/(d-\vert\alpha\vert)}\\
\ge& f(\mathbf{0})- \sum_{\alpha\in \Delta(f)}|\alpha|\cdot \frac{|f_{\alpha}|}{d} -\sum_{f_{d,i}<0} |f_{d,i}|-
\sum_{\alpha\in \Delta(f)^{<d}}(d-|\alpha|)\cdot\frac{|f_{\alpha}|}{d}\\
=&f(\mathbf{0})-\sum_{\alpha\in \Delta(f)} |f_{\alpha}|  -\sum_{f_{d,i}<0} |f_{d,i}|\\
=& f_{\operatorname{tr},\mathbf{N}}.
\end{align*}

(2) Suppose $(\mathbf{z},\lambda)$ satisfies (\ref{lw5}). Since we are trying to maximize (\ref{obj}), we may as well assume each
$\lambda_i$ is chosen as small as possible, i.e., $\lambda_i = \max\{ 0, \sum_{\alpha\in \Delta(f)} z_{\alpha,i}-f_{d,i}\}$.
Since we are also assuming $f_{d,i}\le 0$, this means $\lambda_i = \sum_{\alpha\in \Delta(f)} z_{\alpha,i}-f_{d,i}$. Then
\begin{align*}
&f(\mathbf{0})-\sum_{i=1}^n \lambda_i- \sum_{\alpha\in\Delta(f)^{<d}} (d-\vert\alpha\vert)\left[
\left(\frac{f_\alpha}{d}\right)^{d}\,\left(\frac{\alpha}{\mathbf{z}_\alpha}\right)^\alpha
\right]^{1/(d-\vert\alpha\vert)}\\
&=f(\mathbf{0})-\sum_{i=1}^n(\sum_{\alpha\in \Delta(f)}z_{\alpha,i}-f_{d,i})- \sum_{\alpha\in\Delta(f)^{<d}} (d-\vert\alpha\vert)\left[
\left(\frac{f_\alpha}{d}\right)^{d}\,\left(\frac{\alpha}{\mathbf{z}_\alpha}\right)^\alpha
\right]^{1/(d-\vert\alpha\vert)}\\
&=f(\mathbf{0})-\sum_{\alpha\in \Delta(f)^{<d}}(\sum_{i=1}^n z_{\alpha,i}+(d-\vert\alpha\vert)\left[
\left(\frac{f_\alpha}{d}\right)^{d}\,\left(\frac{\alpha}{\mathbf{z}_\alpha}\right)^\alpha
\right]^{1/(d-\vert\alpha\vert)})\\
&-\sum_{\alpha\in \Delta(f)^{=d}}(\sum_{i=1}^n z_{\alpha,i})-\sum_{i=1}^n |f_{d,i}|.
\end{align*}
We \it claim \rm that, for each $\alpha \in \Delta(f)^{<d}$, the minimum value of
\[
	\sum_{i=1}^n z_{\alpha,i}+(d-\vert\alpha\vert)\left[
	 \left(\frac{f_\alpha}{d}\right)^{d}\,\left(\frac{\alpha}{\mathbf{z}_\alpha}\right)^\alpha
	\right]^{1/(d-\vert\alpha\vert)}
\]
subject to $z_{\alpha,i}\ge 0$ and $z_{\alpha,i}=0$ iff $\alpha_i=0$ is $|f_{\alpha}|$; and that, for each
$\alpha \in \Delta(f)^{=d}$, the minimal value of
\[
	\sum_{i=1}^n z_{\alpha,i}
\]
subject to $\left(\frac{f_\alpha}{d}\right)^{d}\left(\frac{\alpha}{\mathbf{z}_\alpha}\right)^\alpha \, = \, 1$,
$z_{\alpha,i}\ge 0$ and $z_{\alpha,i}=0$ iff $\alpha_i=0$ is also equal to $|f_{\alpha}|$.
It follows from the claim that the maximum value of (\ref{obj}) is equal to
\[
	s(f,\mathbf{g}) = f(\mathbf{0})-\sum_{\alpha\in \Delta(f)}|f_{\alpha}|-\sum_{i=1}^n |f_{d,i}| = f_{\operatorname{tr},\mathbf{N}}.
\]
In proving the claim, one can reduce first to the case where each $\alpha_i$ is strictly positive. The claim, in this case,
is a consequence of the following lemma.
\end{proof}
\begin{lemma}
Suppose $\alpha_i>0$, $i=1,\dots,n$.

(1) For $|\alpha|<d$, the minimum value of
\[
	\sum_{i=1}^n z_i + (d-|\alpha|)\left[(\frac{f_{\alpha}}{d})^d\prod_{i=1}^n \frac{\alpha_i^{\alpha_i}}{z_i^{\alpha_i}}\right]^{1/(d-|\alpha|)}
\]
on the set $(0,\infty)^n$ is equal to $|f_{\alpha}|$. The minimum occurs at $z_i = \alpha_i\cdot \frac{|f_{\alpha}|}{d}$, $i=1,\dots,n$.

(2) For $|\alpha|=d$, the minimum value of $\sum_{i=1}^n z_i$ subject to $z_i>0$ and
$(\frac{|f_{\alpha}|}{d})^d\cdot \prod_{i=1}^n\frac{\alpha_i^{\alpha_i}}{z_i^{\alpha_i}} = 1$ is equal to $|f_{\alpha}|$.
The minimum occurs at $z_i = \alpha_i\cdot \frac{|f_{\alpha}|}{d}$, $i=1,\dots,n$.
\end{lemma}
\begin{proof} the optimization problem in (1) is equivalent to the problem of minimizing the function
$\sum_{i=1}^{n+1} z_i$ subject to $z_i>0$ and
$(\frac{|f_{\alpha}|}{d})^d\cdot \prod_{i=1}^n\frac{\alpha_i^{\alpha_i}}{z_i^{\alpha_i}}\cdot \frac{(d-|\alpha|)^{d-|\alpha|}}{z_{n+1}^{d-|\alpha|}} = 1$.
In this way, (1) reduces to (2). The proof of (2) is straightforward, e.g., making the change in variables
$w_i = \frac{z_i d}{\alpha_i |f_{\alpha}|}$ we are reduced to minimizing $\sum_{i=1}^n \alpha_iw_i$ subject to
$\prod_{i=1}^n w_i^{\alpha_i}=1$. Using the relation between the arithmetic and geometric mean yields
\[
	\frac{\sum_{i=1}^n \alpha_iw_i}{d}\ge \root d \of{\prod_{i=1}^n w_i^{\alpha_i}}= 1,
\]
i.e., $\sum_{i=1}^n \alpha_iw_i \ge d$. On the other hand, if we take $w_i=1$, then $\sum_{i=1}^n \alpha_iw_i = |\alpha|=d$.
Thus the minimum occurs at $w_i=1$, i.e., $z_i = \alpha_i \frac{|f_{\alpha}|}{d}$, $i=1,\dots,n$, and the minimum value of
$\sum_{i=1}^n z_i$ is $\sum_{i=1}^n \alpha_i \frac{|f_{\alpha}|}{d} = |\alpha|\frac{|f_{\alpha}|}{d} = |f_{\alpha}|$.
\end{proof}
\begin{rem} \

(1) Suppose $I_1,\dots,I_{\ell}$ and $J_1,\dots,J_m$ are partitions of $\{ 1,\dots,n\}$ with $I_1,\dots,I_{\ell}$ finer
than $J_1,\dots,J_m$,
\[
	G(\lambda) = f-\sum_{p=1}^{\ell}\lambda_p \left(1-\sum_{i\in I_p} (\frac{x_i}{N_i})^d\right), \ H(\mu) =
	f-\sum_{q=1}^m \mu_q\left(1-\sum_{i\in J_q}(\frac{x_i}{N_i})^d\right).
\]
One checks that if $\mu_q = \sum_{I_p\subseteq J_q} \lambda_p,$ then $G(\lambda)_{\operatorname{gp}} \le H(\mu)_{\operatorname{gp}}$.
It follows that $s(f,\mathbf{g}) \le s(f,\mathbf{h})$ where
\[
	\mathbf{g} = \left(1- \sum_{i\in I_1}(\frac{x_i}{N_i})^d,\dots,1- \sum_{i\in I_{\ell}}(\frac{x_i}{N_i})^d\right),
	\quad \mathbf{h} = \left(1- \sum_{i\in J_1}(\frac{x_i}{N_i})^d,\dots,1- \sum_{i\in J_m}(\frac{x_i}{N_i})^d\right).
\]

(2) Similarly, one checks that if
\[
	H(\lambda)= f-\lambda(1-\sum_{i=1}^n (\frac{x_i}{N_i})^d), \ I(\mu) = f-\sum_{j=1}^n \mu_j(\frac{1}{n}-(\frac{x_i}{N_i})^d)
\]
where $\mu_j = \lambda$, $j=1,\dots,n$, then $H(\lambda)=I(\mu)$. It follows that $s(f,\mathbf{h}) \le s(f,\mathbf{i})$ where
\[
	\mathbf{h}=\left(1-\sum_{i=1}^n(\frac{x_i}{N_i})^d\right),\ \mathbf{i}=\left(\frac{1}{n}-(\frac{x_1}{N_1})^d,\dots,\frac{1}{n}-(\frac{x_n}{N_n})^d\right).
\]

(3) In particular, (1) and (2) imply
\[
	s(f,\mathbf{g})\le s(f,\mathbf{h}) \le s(f,\mathbf{i}),
\]
with $\mathbf{g}$ as in Theorem \ref{hypercube case}, $\mathbf{h}$ and $\mathbf{i}$ as in (2). Observe also that, by
Theorem \ref{hypercube case}, $f_{\operatorname{tr},\mathbf{N}} \le s(f,\mathbf{g})$ and
$f_{\operatorname{tr},\mathbf{N}/\root d \of{n}} \le s(f,\mathbf{i})$ with equality holding if $f_{d,i}\le 0$, $i=1,\dots,n$.
This clarifies to some extent an observation made in \cite[Section 3]{gha-lass-mar}.
\end{rem}

\end{document}